\documentclass[10pt,reqno]{amsart}
\usepackage{amsmath,multicol}
\usepackage{amsthm}
\usepackage{amssymb}
\usepackage{amsfonts}
\usepackage{latexsym}
\usepackage{graphicx}
\usepackage{cite, hyperref, xcolor}
\usepackage{subcaption, url, cite, enumitem}
\usepackage[font=footnotesize,labelfont=bf]{caption}
\usepackage[sc]{mathpazo}

\usepackage[T1]{fontenc}

\renewcommand{\phi}{\varphi}

\newcommand{\Z}{\mathbb{Z}}

\newcommand{\F}{\mathbb{F}}

\renewcommand{\vec}[1]{{\bf #1}}

\newcommand{\diag}{\operatorname{diag}}

 \theoremstyle{plain}
\newtheorem{theorem}[equation]{Theorem}
\newtheorem{lemma}[equation]{Lemma}

\theoremstyle{remark}

\theoremstyle{definition}

\allowdisplaybreaks
\begin{document}
\title{Supercharacters, elliptic curves, and the sixth moment of Kloosterman sums}
%\date{\today}

\author{Stephan Ramon Garcia}
\address{Department of Mathematics, Pomona College, 610 N. College Ave., Claremont, CA 91711} 
\email{stephan.garcia@pomona.edu}
\urladdr{\url{http://pages.pomona.edu/~sg064747}}

\author{George Todd}
\address{Department of Mathematics, Union College, Bailey Hall 202, Schenectady, NY 12308}
\email{toddg@union.edu}

\begin{abstract}
We connect the sixth power moment of Kloosterman sums to elliptic curves.  This yields an elementary proof that $K_u$ with $p\nmid u$ are $O(p^{2/3})$.
\end{abstract}

\thanks{First author supported by a David L. Hirsch III and Susan H. Hirsch Research Initiation Grant and the Institute for Pure and Applied Mathematics (IPAM) Quantitative Linear Algebra program.}

\maketitle

\section{Introduction}
Let $e_p(x) = \exp(2\pi i x / p)$, in which $p$ is an odd prime.
A \emph{Kloosterman sum} is
\begin{equation*}
K(a,b) = \sum_{x =1}^{p-1} e_p(ax + bx^{-1}),
\end{equation*}
in which $x^{-1}$ is the inverse of $x$ modulo $p$.
Kloosterman sums are real and satisfy 
$K(a,b) = K(1,ab)$ if $p \nmid a$.  Consequently, we write
$K_u = K(1,u)$.
The celebrated \emph{Weil bound} asserts that
$|K(u)| \leq 2 \sqrt{p}$ for $p \nmid u$ \cite{Weil}.  
%The standard proof
%involves the rationality of a certain zeta function
%and a version of the Riemann hypothesis;
%see \cite[Thm.~11.11]{Iwaniec}.
%Kloosterman obtained $K_u = O(p^{3/4})$ through elementary
%means \cite{Kloosterman}; see also \cite{DRHB, CKS}.  

The first several \emph{power moments} 
\begin{equation*}
V_n(p) = \sum_{u=1}^{p-1} K_u^n
\end{equation*}
of the Kloosterman sums are
\begin{align*}
V_1(p) &= 1,
&V_2(p) &= p^2-p-1,\\
V_3(p) &= \left( \frac{p}{3} \right) p^2 + 2p + 1,
&V_4(p) &= 2p^3 - 3p^2 - 3p -1,
\end{align*}
in which $(\frac{\cdot}{p})$ denotes the Legendre symbol modulo $p$.
See \cite{CKS} for simple proofs of the preceding and \cite{Chavez} for
additional mixed-moment evaluations.
An expression for $V_5(p)$ was found by Livne \cite{Livne}
and by Peters, Top, and van der Vlugt \cite{Peters}:
\begin{equation*}\qquad\qquad
V_5(p) = \left( \frac{p}{3} \right)4p^3 + (a_p + 5)p^2 + 4p + 1,\qquad p> 5,
\end{equation*}
in which $|a_p| < 2p$ depends upon $p$;
see \cite[p.~1234]{XiYi} or \cite[p.~112]{ZhangHan} for details.  In the early 1930s,
H.~Sali\'e \cite{Salie} and H.~Davenport \cite{Davenport} proved that
$V_6(p) = O(p^4)$.
A precise evaluation of $V_6(p)$ was obtained in 2001 by 
Hulek, Spandaw, van Geemen, and van Straten \cite{Hulek}.  They showed that
\begin{equation}\label{eq:V6Old}\qquad\qquad
V_6(p) = 5p^4 - 10p^3 - (b_p + 9)p^2 - 5p - 1, \qquad p>7,
\end{equation}
in which $b_p$ is an integer with $|b_p| < 2p^{3/2}$ that is derived from the Dedekind eta function.
Consequently, 
\begin{equation*}
V_6(p) = 5p^4 + O(p^{7/2})
\end{equation*}
and hence $K_u = O(p^{2/3})$.
In 2010, Evans conjectured formulas for $V_7(p)$ and $V_8(p)$ \cite{Evans, Evans2}, which were ultimately proved by Yun \cite{Yun2} by attaching Galois representations to Kloosterman sums. It should be noted that the fifth through eighth moments can be related to Hecke eigenvalues;
see Section \ref{FutureWork}. Exact formulas for $V_n(p)$ for $n \geq 9$ appear difficult to obtain.

Our main result is a formula that relates the sixth power moment of Kloosterman
sums to elliptic curves.  
This particular connection appears novel for $p \geq 5$,
although for $p=2$ and $p=3$ some links between Kloosterman sums and elliptic
curves have been discovered \cite{Lisonek}.

\begin{theorem}\label{Theorem:Main}
For $p\geq 5$,
\begin{equation}\label{eq:V6}
V_6(p)
= 
4 p^4 -8p^3+ \left[4 \left( \frac{p}{3} \right)+2\right] p^2-5 p-1  + p^2 \sum_{\substack{k=2\\k\neq 9}}^{p-1} (a_p(E_k)+1)^2,
\end{equation}
in which $a_p(E_k)$ denotes the Frobenius trace of the elliptic curve
\begin{equation*}
E_k(\F_p) = \big\{ (x,y) \in \F_p^2\, : \,y^2 = 4k x^3 +  (k^2 - 6 k  -3) x^2 + 4x\big\}.
\end{equation*}
\end{theorem}

The restriction $k\neq 1,9$ is natural since
these yield the non-elliptic curves
$y^2 = 4x(x-1)^2$ and $y^2 = 4x(3x+1)^2$,
respectively.  For $p=5,7$, we interpret these restrictions modulo $p$.
For example, if $p=7$, then the terms corresponding to $k=1$ and $k=2$ are omitted in \eqref{eq:V6}.
For $k\neq 1,9$, the number of points on $E_k(\F_p)$, including the point at infinity, is
\begin{equation*}
|E_k(\F_p)| = p+1 - a_p(E_k),
\end{equation*}
in which the Frobenius trace $a_p(E_k)$ satisfies \emph{Hasse's inequality} \cite{Hasse}
\begin{equation*}
|a_p(E_k)| \leq 2\sqrt{p}.
\end{equation*}
An elementary inductive proof of this
was found by Y.~Manin \cite{Manin}.
It is simple enough that it appeared in the American Mathematical
Monthly in 2008 \cite{ChahalMonthly}.

To prove Theorem \ref{Theorem:Main},
we employ basic supercharacter theory to 
realize Kloosterman sums as eigenvalues of a certain matrix whose entries encode combinatorial information about a
certain group action.  
This completely elementary, linear-algebraic perspective provides a convenient method for keeping track
of various expressions that arise throughout our computations.  
This approach was first undertaken to study Ramanujan sums \cite{RSS};
see also \cite{GHM}.

As a consequence of Theorem \ref{Theorem:Main}, we obtain an elementary proof that
\begin{equation}\label{eq:BarrierBound}
|K_u| \leq 1.43 p^{2/3}
\end{equation}
whenever $p \nmid u$.
In particular, this breaks the ``$O(p^{3/4})$ barrier,'' which folklore suggested
cannot be passed without deep techniques or difficult point-counting arguments.
To obtain \eqref{eq:BarrierBound}, use Hasse's inequality in \eqref{eq:V6} and compute
\begin{align*}
V_6(p)
&\leq 4 p^4 -8p^3+ \left[4 \left( \frac{p}{3} \right)+2\right] p^2-5 p-1  + p^2 (p-3)
(4p +4\sqrt{p}+1)\\
&= 8 p^4 +4 p^{7/2}-19 p^3 -12 p^{5/2}+ \left[4  \left( \frac{p}{3} \right)-1\right] p^2-5 p-1 \\
&\leq 8 p^4 +4 p^{7/2}-19 p^3 -12 p^{5/2}+ 3 p^2-5 p-1 \\
&\leq 8.5p^4
\end{align*}
for $p\geq 5$.  Taking sixth roots and verifying the cases $p=2,3$
yields \eqref{eq:BarrierBound}.

%A bound $a_p(E_k) = O(p^{\alpha})$ results in 
%$K_u = O(p^{(2\alpha+3)/6})$ for $p \nmid u$.  In particular,
%$\alpha < \frac{3}{4}$ results in a bound on Kloosterman sums that beats $O(p^{3/4})$.
%Is there is very short proof of a ``weak Hasse bound'' sufficient to break the $O(p^{3/4})$ barrier?

\smallskip\noindent\textbf{Acknowledgments}:  We thank Terence Tao for many helpful suggestions.

\section{Kloosterman sums as supercharacters}
The theory of supercharacters was introduced in 2008 by P.~Diaconis and I.M.~Isaacs \cite{Diaconis},
building upon previous work of C.~Andr\'e \cite{Andre1} on the representation theory of unipotent matrix
groups over finite fields.  We are concerned only with the special case in which the underlying group is abelian,
for which the details are much simpler.
A variety of exponential sums that are relevant to the theory of numbers can be realized
as supercharacters on abelian groups.    The following setup is from \cite{SESUP}.

Let $\Gamma$ be a subgroup of $GL_d( \Z/n\Z)$ that is closed under the transpose operation and
let $X_1,X_2,\ldots,X_N$ denote the orbits in $G=(\Z/n\Z)^d$ under the action of $\Gamma$.  The functions
\begin{equation}\label{eq:Supercharacter}
\sigma_i(\vec{y}) = \sum_{\vec{x} \in X_i} e \left( \frac{ \vec{x} \cdot \vec{y} } {n} \right),
\end{equation}
in which $\vec{x}\cdot \vec{y}$ denotes the formal dot product of two elements of $(\Z/n\Z)^d$
and $e(x) = \exp(2\pi i x)$,
are \emph{supercharacters} on $(\Z/n\Z)^d$ and the sets $X_i$
are \emph{superclasses}.  One can show that supercharacters are constant
on superclasses, so we may write $\sigma_i(X_j)$ without
confusion.  The $N \times N$ matrix
\begin{equation}\label{eq:USCT}
U = \frac{1}{\sqrt{n^d}} \left[  \frac{   \sigma_i(X_j) \sqrt{  |X_j| }}{ \sqrt{|X_i|}} \right]_{i,j=1}^N
\end{equation}
is complex symmetric (i.e., $U = U^{\mathsf{T}}$) and unitary \cite[Lem.~1]{SESUP}, \cite[Sect.~2.1]{RSS}.
It represents, with respect to a particular
orthonormal basis, the restriction of the discrete Fourier transform (DFT) to the subspace of $L^2(G)$
that consists of functions that are constant on each $\Gamma$-orbit in $G$.

The following lemma identifies the set of matrices that are diagonalized
by the unitary matrix \eqref{eq:USCT} as the span of a certain family of normal matrices that contain 
combinatorial information about the group action.  The proof is completely elementary and is 
similar to the corresponding result from classical character theory
\cite[Section 33]{CR62}.  In fact, the matrices \eqref{eq:U}, \eqref{eq:D}, \eqref{eq:T} below and their properties
can be obtained with classical character theory in a more contrived, tedious, and long-winded manner \cite[Lem.~3.1]{CKS}.
A more general version of this lemma, in which $G$ need not be abelian, is \cite[Thm.~4.2]{RSS}.
The simple version that we present below is \cite[Thm.~1]{SESUP}.

\begin{lemma}\label{Lemma:SESUP}
Let $\Gamma = \Gamma^{\mathsf{T}}$ be a subgroup of $GL_d(\Z/n\Z)$, let $\{X_1,X_2,\ldots,X_N\}$
denote the set of $\Gamma$-orbits in $G = (\Z/n\Z)^d$ induced by the action of $\Gamma$, 
and let $\sigma_1,\sigma_2,\ldots,\sigma_N$
denote the corresponding supercharacters \eqref{eq:Supercharacter}.
For each fixed $\vec{z}$ in $X_k$, let $c_{i,j,k}$ denote the number of solutions $(\vec{x}_i,\vec{y}_j) \in X_i \times X_j$ to the equation $\vec{x}+\vec{y} = \vec{z}$.
\begin{enumerate}\addtolength{\itemsep}{0.5\baselineskip}
\item $c_{i,j,k}$ is independent of the representative $\vec{z}$ in $X_k$ that is chosen.
\item The identity
\begin{equation*}
\sigma_i(X_{\ell}) \sigma_j(X_{\ell}) = \sum_{k=1}^N c_{i,j,k} \sigma_k(X_{\ell})
\end{equation*}
holds for $1\leq i,j,k,\ell \leq N$.
\item The matrices $T_1,T_2,\ldots,T_N$, whose entries are given by
\begin{equation*}
[T_i]_{j,k} = \frac{ c_{i,j,k} \sqrt{ |X_k| } }{ \sqrt{ |X_j|} },
\end{equation*}
each satisfy 
\begin{equation*}
T_i U = U D_i,
\end{equation*}
in which 
\begin{equation*}
D_i = \operatorname{diag}\big(\sigma_i(X_1), \sigma_i(X_2),\ldots, \sigma_i(X_N) \big).
\end{equation*}
In particular, the $T_i$ are simultaneously unitarily diagonalizable.

\item Each $T_i$ is a normal matrix (i.e., $T_i^*T_i = T_i T_i^*$) and
the set $\{T_1,T_2,\ldots,T_N\}$ forms a basis for the commutative algebra of all $N \times N$ complex matrices 
$T$ such that $U^*TU$ is diagonal.
\end{enumerate}
  \end{lemma}

Let $p$ be an odd prime.  Then the action of the diagonal matrix group
\begin{equation*}
\Gamma = \{ \operatorname{diag}(u,u^{-1}) : u \in \F_p^{\times} \}
\end{equation*}
on the additive group $G = \F_p^2$ induces a supercharacter theory 
that is related to Kloosterman sums.
There are $N = p+2$ superclasses (that is, $\Gamma$-orbits in $G$):
\begin{equation*}
\begin{array}{rcl}
X_1 &=& \big\{ (x,x^{-1}) : x \in \F_p^{\times}\big\},\\[3pt]
X_2 &=& \big\{ (x,2x^{-1}) : x \in \F_p^{\times}\big\},\\[3pt]
&\vdots& \\
X_{p-1} &=& \big\{ (x,(p-1)x^{-1}) : x \in \F_p^{\times}\big\},\\[3pt]
X_p &=& \big\{(0,u) : u \in \F_p^{\times} \big\} ,\\[3pt]
X_{p+1} &=& \big\{(u,0): u \in \F_p^{\times} \big\} ,\\[3pt]
X_{p+2} &=& \big\{(0,0)\big\}.
\end{array}
\end{equation*}
If $1 \leq i,j \leq p-1$, then we select the representative $\vec{y} = (1,j) \in X_j$ and compute:
\begin{align*}
\sigma_i(X_j)
&= \sum_{\vec{x} \in X_i} e_p( \vec{x} \cdot \vec{y})
= \sum_{u =1}^{p-1} e_p\big ((u,iu^{-1}) \cdot (1,j) \big) \\
&= \sum_{u=1}^{p-1} e_p(u + iju^{-1}) = K_{ij}.
\end{align*}
A few more computations complete the supercharacter table (Table \ref{Table:SCT}).

\begin{table}[b]
\begin{equation*}\small
\begin{array}{|c|cccc|cc|c|}
\hline
(\Z/p\Z)^2 & X_1 & X_2 & \cdots & X_{p-1} & X_p & X_{p+1} & X_{p+2} \\[2pt] 
\Gamma & (1,1) & (1,2) & \cdots & (1,p-1) & (0,1) & (1,0) & (0,0) \\[2pt]
\# & p-1 & p-1 & \cdots & p-1 & p-1 & p-1 & 1  \\[2pt]
\hline
\sigma_1 & K_1 & K_2 & \cdots & K_{p-1} & -1 & -1 & p-1 \\[2pt]
\sigma_2 & K_2 & K_4 & \cdots & K_{2(p-1)} & -1 & -1 & p-1 \\[2pt]
\vdots & \vdots & \vdots & \ddots & \vdots & \vdots & \vdots & \vdots \\[2pt]
\sigma_{p-1} & K_{p-1} & K_{2(p-1)} & \cdots & K_{(p-1)^2} & -1 & -1 & p-1 \\[2pt]
\hline
\sigma_p & -1 & -1 & \cdots & -1 & p-1 & -1 & p-1 \\[2pt]
\sigma_{p+1} & -1 & -1 & \cdots & -1 & -1 & p-1 & p-1 \\[2pt]
\hline
\sigma_{p+2} & 1 & 1 & \cdots & 1 & 1 & 1 & 1 \\[2pt]
\hline
\end{array}
\end{equation*}
\caption{Supercharacter table for the action of $\Gamma =\{ \operatorname{diag}(u,u^{-1}) : u \in \F_p^{\times} \}$
on $G = \F_p^2$.}
\label{Table:SCT}
\end{table}

The formula \eqref{eq:USCT} provides 
the $(p+2) \times (p+2)$ real-symmetric unitary matrix
\begin{equation}\label{eq:U}\small
U:=
\frac{1}{p}
\left[
\begin{array}{cccc|cc|c}
K_1&K_2&\cdots&K_{p-1}&-1&-1&\sqrt{p-1}\\[2pt]
K_2&K_4&\cdots&K_{2(p-1)}&-1&-1&\sqrt{p-1}\\[2pt]
\vdots&\vdots&\ddots&\vdots&\vdots&\vdots&\vdots\\[2pt]
K_{p-1}&K_{2(p-1)}&\cdots&K_{(p-1)^2}&-1&-1&\sqrt{p-1}\\[2pt]
\hline
-1&-1&\cdots&-1& p-1&-1&\sqrt{p-1}\\[2pt]
-1&-1&\cdots&-1&-1& p-1&\sqrt{p-1}\\[2pt]
\hline
\sqrt{p-1}&\sqrt{p-1}&\cdots &\sqrt{p-1}&\sqrt{p-1}&\sqrt{p-1}&1\\[2pt]
\end{array}
\right];
\end{equation}
this is \cite[eq.~3.13]{CKS}.
Define $D = D_1$ and $T = T_1$ as in  Lemma \ref{Lemma:SESUP}.
These are
\begin{equation}\label{eq:D}
D = \diag(K_{1},K_{2},\ldots,K_{(p-1)},-1,-1,p-1)  
\end{equation}
and
\begin{equation}\label{eq:T}
T =    
\left[
\begin{array}{cccc|cc|c}
t_{1,1}&t_{1,2}&\cdots&t_{1,p-1}&0&0&\sqrt{p-1}\\
t_{2,1}&t_{2,2}&\cdots&t_{2,p-1}&1&1&0\\
\vdots&\vdots&\ddots&\vdots&\vdots&\vdots&\vdots\\
t_{p-1,1}&t_{p-1,2}&\cdots&t_{p-1,p-1}&1&1&0\\
\hline
0&1&\cdots&1&0&1&0\\
0&1&\cdots&1&1&0&0\\
\hline
\sqrt{p-1}&0&\cdots &0&0&0&0\\
\end{array}
\right],
\end{equation}
in which
\begin{equation}\label{eq-cijk}
t_{i,j} = 1 + \left( \frac{f_j(i)}{p} \right)
\end{equation}
and
\begin{equation}\label{eq:Beta}
f_j(x) = x^2 -2(j+1)x + (j-1)^2.
\end{equation}
The symmetry $f_j(i) = f_i(j)$ will be important later.
The matrix $T$ is \cite[eq.~3.12]{CKS} and it satisfies
$T = UDU$
by Lemma \ref{Lemma:SESUP} since $U^* = U$.
These matrices and the diagonalization above can also be derived, with more effort, using
classical character theory \cite{CKS}.

Before proceeding, a brief explanation of \eqref{eq:Beta} is in order.  
Let $1 \leq i,j\leq p-1$.
In the notation of Lemma \ref{Lemma:SESUP},
$t_{i,j} = c_{1,i,j}$ is the number of solutions to $\vec{x} + \vec{y} = \vec{z}$, in which 
$\vec{z} = (1,j) \in X_j$ is fixed, $\vec{x} = (x,x^{-1}) \in X_1$ and $\vec{y} = (y,iy^{-1})\in X_i$.  This yields
the system
\begin{equation*}
(x,x^{-1}) + (y,iy^{-1}) = (1,j).
\end{equation*}
Since $x=1$ implies $y=0$, we may assume that $x \neq 1$.
The first equation $x+y=1$ suggests the substitution $y = 1-x$.  The second equation then yields
\begin{equation*}
jx^2 + (i-j-1)x + 1 = 0,
\end{equation*}
the discriminant of which is $f_j(i)$.  This establishes \eqref{eq:Beta}.

\section{Proof of Theorem \ref{Theorem:Main}}\label{Section:Proof}
We prove the desired identity \eqref{eq:V6} by computing
$[T^4]_{1,1} = [UD^4 U]_{1,1}$ in two different ways.
The evaluation of $[UD^4 U]_{1,1}$ is relatively simple and involves
$V_6(p)$; we save this for later.  To compute $[T^4]_{1,1}$ requires more work.
Some of the expressions that arise involve
Frobenius traces of certain elliptic curves over $\F_p$.

\begin{lemma}
For $k=1,2,\ldots,p-1$,
\begin{equation}\label{eq:LegendreSum}
\sum_{x=0}^{p-1} \left( \frac{ f_k(x)}{p} \right) = - 1  
\qquad \text{and} \qquad
\sum_{x=0}^{p-1} \left( \frac{ f_k(x)}{p} \right)^2 
= p-1-\left( \frac{k}{p}\right)  .
\end{equation}
\end{lemma}

\begin{proof}
Since $f_k(x+k+1) = x^2 - 4k$ and $p \nmid 4k$, the first equation in \eqref{eq:LegendreSum}
follows from \cite[Ex.~8, p.~63]{Ireland}.  The number of solutions to $x^2 - 4k \equiv 0 \pmod{p}$ 
is $1+ (k/p)$, from which the second equation in \eqref{eq:LegendreSum} follows.
\end{proof}

Let
\begin{equation}\label{eq:Epsilon}
\epsilon_k = \sum_{x=0}^{p-1} \left( \frac{f_1(x) f_k(x)}{p} \right).
\end{equation}
In what follows, the Legendre symbol $(p/3)$ occurs frequently.  Consequently,
we adopt the shorthand $\ell_p = (p/3)$.  Since $p\geq 5$, it follows that $\ell_p^2 = 1$.

The quadratic formula confirms that
\begin{equation*}
f_1(x) = x^2 - 4x \qquad \text{and} \qquad f_k(x) = x^2 - 2(k+1)x + (k-1)^2
\end{equation*}
share a common root if and only if $k = 1$ or $k= 9$.  This causes some minor complications
later on when we attempt to write \eqref{eq:Epsilon} in terms of Frobenius traces.
We therefore evaluate $\epsilon_1$ and $\epsilon_9$ explicitly here.

\begin{lemma}\label{Lemma:19}
$\epsilon_1 = p-2$ and $\epsilon_9 = -1-\ell_p$.
\end{lemma}

\begin{proof}
Since $f_1(x) = x(x-4)$ has two distinct roots modulo $p$,
\begin{equation*}
\epsilon_1 = \sum_{x=0}^{p-1} \left( \frac{f_1(x)^2}{p} \right)  = \sum_{x=0}^{p-1} \left( \frac{f_1(x)}{p} \right)^2 = p-2.
\end{equation*}
Since $f_1$ and $f_k$ share the common factor $x-4$, \eqref{eq:LegendreSum} ensures that
\begin{align*}
\epsilon_9 &= \sum_{x=0}^{p-1} \left( \frac{f_1(x)}{p} \right) \left( \frac{f_9(x)}{p} \right) = \sum_{x=0}^{p-1} \left( \frac{x(x-4)}{p} \right) \left( \frac{(x-4)(x-16)}{p} \right) \\
&= \sum_{x=0}^{p-1} \left( \frac{x(x-16)}{p} \right) - \left(\frac{-48}{p} \right) = \sum_{x=0}^{p-1} \left( \frac{x^2 - 16x + 64 - 64}{p} \right) - \left(\frac{p}{3}\right) \\
&= \sum_{x=0}^{p-1} \left( \frac{(x - 8)^2 - 64}{p} \right)  - \left(\frac{p}{3}\right)
= -1 - \left(\frac{p}{3}\right) = - 1 - \ell_p.\qedhere
\end{align*}
\end{proof}

\begin{lemma}\label{Lemma:EpsilonSum}
$\displaystyle\sum_{k=2}^{p-1} \epsilon_k = 4 + \ell_p - p$.
\end{lemma}

\begin{proof}
The symmetry $f_k(x) = f_x(k)$, \eqref{eq:LegendreSum}, and quadratic reciprocity imply that
\begin{align*}
\sum_{k=1}^{p-1} \epsilon_k
&=  \sum_{k=1}^{p-1} \sum_{x = 0}^{p-1}
\left( \frac{f_1(x)f_k(x)}{p} \right) 
=\sum_{x = 0}^{p-1}\left( \frac{f_1(x)}{p} \right)\sum_{k=1}^{p-1}\left( \frac{f_x(k)}{p} \right) \\
&= \sum_{x = 0}^{p-1}\left( \frac{f_1(x)}{p} \right) \left[-1-\left( \frac{f_x(0)}{p} \right)\right] 
= \sum_{x = 0}^{p-1}\left( \frac{f_1(x)}{p} \right) \left[-1-\left( \frac{(x-1)^2}{p} \right)\right] \\
&=\left( \frac{f_1(1)}{p} \right) - 2 \sum_{x = 0}^{p-1}\left( \frac{f_1(x)}{p} \right) 
= 2 + \left( \frac{-3}{p} \right) = 2 + \left( \frac{p}{3} \right). 
\end{align*}
Subtract $\epsilon_1 = p-2$ and obtain the desired result.
\end{proof}

The following is a special case of a formula due to K.S.~Williams \cite{WilliamsPAMS}.  
In this instance, it concerns the birational equivalence between a quartic and a cubic elliptic curve.  
%The reference \cite{WilliamsPAMS} states the result without a direct proof and refers 
%to \cite{WilliamsPacific} for the lengthy and notationally cumbersome 
%proof of a more general statement.  
We provide an independent proof of the relevant case.

\begin{lemma}\label{Lemma:Transform}
Let $p$ be prime,
\begin{equation*}
D = B^2 - 4 C,\qquad
d = b^2 - 4c,\quad\text{and}\quad
\delta = 4C - 2bB + 4c.
\end{equation*} 
If $B-b \neq 0$ and the polynomials $x^2 + bx + c$ and $x^2 + Bx + C$ share no common roots modulo $p$, then
\begin{equation*}
\sum_{x=0}^{p-1} \left( \frac{(x^2+bx+c)(x^2+Bx+C)}{p} \right)
= -1+ \sum_{x=0}^{p-1} \left( \frac{x(Dx^2 + \delta x + d)}{p}\right) .
\end{equation*}
\end{lemma}

\begin{proof}
For $x^2 + Bx + C \neq 0$, define the $\F_p$-valued function
\begin{equation*}
\theta(x) = \frac{x^2 + bx + c}{x^2 + Bx + C}.
\end{equation*}
Suppose that $y \neq 0$, then since $x^2 + bx + c$ and $x^2 + Bx + C$ share no common roots modulo $p$, the number of solutions to $\theta(x) = y$ equals the number
of solutions to
\begin{equation}\label{eq:QuadraticY}
    (y - 1)x^2 + (By - b)x + (Cy - c) = 0.
\end{equation}
If $y \neq 1$, then the number of solutions to \eqref{eq:QuadraticY} is
\begin{equation*}
1+ \left( \frac{(By - b)^2 - 4(y - 1)(Cy - c)}{p} \right)
= 1 + \left(\frac{Dy^2 + \delta y + d}{p} \right).
\end{equation*}
If $y = 1$, then \eqref{eq:QuadraticY} is 
\begin{equation*}
(B - b)x + (C - c) = 0,
\end{equation*}
which has exactly
\begin{equation*}
1 = \left( \frac{(B - b)^2}{p} \right)
= \left(\frac{1(D (1)^2 + \delta (1) + d)}{p} \right)
\end{equation*}
solutions since $B-b \neq 0$.  Then
\begin{align*}
&\sum_{x\in \F_p} \left( \frac{(x^2+bx+c)(x^2+Bx+C)}{p} \right) \\
&\qquad= \sum_{\substack{ x \in \F_p \\ x^2+Bx+C \neq 0}} \left( \frac{\theta(x)}{p} \right)  
= \sum_{y \in \F_p\backslash\{0\}} \left( \frac{y}{p} \right) \sum_{ \substack{x \in \F_p \\ \theta(x) =y}}1  \\
&\qquad = \sum_{ \substack{x \in \F_p \\ \theta(x) =1}}1 + \sum_{y \in \F_p\backslash\{0,1\}}  \left( \frac{y}{p} \right)\sum_{ \substack{x \in \F_p \\ \theta(x) =y}}1 \\
&\qquad = \left( \frac{1(D (1)^2 + \delta (1) + d)}{p} \right) 
    + \sum_{y \in \F_p\backslash\{0,1\}}  \left( \frac{y}{p} \right)\left[ 1 + \left(\frac{Dy^2 + \delta y + d}{p} \right) \right] \\
&\qquad = \sum_{y \in \F_p\backslash\{0,1\}} \left( \frac{y}{p} \right)+  \sum_{y \in \F_p\backslash\{0\}} \left(\frac{y(Dy^2 + \delta y + d)}{p} \right) \\
&\qquad =-1 +  \sum_{y \in \F_p} \left(\frac{y(Dy^2 + \delta y + d)}{p} \right) .\qedhere
\end{align*}
\end{proof}

\begin{lemma}\label{Lemma:Bound}
For $k\neq 1,9$, 
\begin{equation}
\epsilon_k = -1 - a_p(E_k),
\end{equation}
in which 
\begin{equation*}
a_p(E_k) = p+1 - |E_k(\F_p)|
\end{equation*}
denotes the Frobenius trace of the elliptic curve
\begin{equation*}
E_k(\F_p) = \big\{ (x,y) \in \F_p^2 : y^2 = g_k(x)\big\},
\end{equation*}
where
\begin{equation*}
g_k(x) =x\big(4k x^2 +  (k^2 - 6 k  -3) x + 4\big).
\end{equation*}
\end{lemma}

\begin{proof}
Suppose that $k\neq 1,9$.
First observe that the discriminant of the quadratic factor of $g_k(x)$ is $(k-9)(k-1)^3$.  Thus, 
$g_k(x)$ has distinct roots and $y^2 = g_k(x)$ defines an elliptic curve $E_k$
over $\F_p$.  Since $f_1(x)$ and  $f_k(x)$ share no common roots in $\F_p$,
we apply Lemma \ref{Lemma:Transform} with
\begin{equation*}
a = 1,\quad b = -4,\quad  c = 0,\quad 
A = 1,\quad  B = -2(k+1),\quad  C = (k-1)^2,
\end{equation*}
so that 
\begin{equation*}
D = 16k, \qquad d = 16, \quad \text{and} \quad \delta = 4 (k^2 - 6 k  -3). 
\end{equation*}
Then 
\begin{align*}
\epsilon_k 
&= \sum_{x = 0}^{p-1} \left( \frac{f_1(x)f_k(x)}{p} \right) 
= \sum_{x = 0}^{p-1} \left( \frac{(x^2-4x)(x^2-2(k+1)x+(k-1)^2)}{p} \right) \\
%&= -1+\sum_{x=0}^{p-1} \left( \frac{4k x^3 +  (k^2 - 6 k  -3) x^2 + 4x}{p} \right) \\
&= -1+\sum_{x=0}^{p-1} \left( \frac{x(4k x^2 +  (k^2 - 6 k  -3) x + 4)}{p} \right) 
= -1+\sum_{x=0}^{p-1} \left( \frac{g_k(x)}{p} \right) \\
&=-p-2 + \left[p+1 +  \sum_{x=0}^{p-1} \left( \frac{g_k(x)}{p} \right) \right]
= -p-2 + |E_k(\F_p)| \\
&= -p-2 + \big(p+1 - a_p(E_k)\big) 
= -1 - a_p(E_k). \qedhere
\end{align*}
\end{proof}

\begin{lemma}\label{Lemma:Entries}
For $k=1,2,\ldots,p+2$,
\begin{equation*}
[T^2]_{1,k}
=
\begin{cases}\displaystyle
3p-6 & \text{if $k=1$},\\[5pt]
p-4+\epsilon_k & \text{if $k=2,\ldots,p-1$},\\[5pt]
p-3 - \ell_p & \text{if $k = p,p+1$},\\[8pt]
(1+\ell_p) \sqrt{p-1} & \text{if $k=p+2$}.
\end{cases}
\end{equation*}
\end{lemma}

\begin{proof}
Use \eqref{eq:T} and \eqref{eq:LegendreSum} to compute 
\begin{align*}
[T^2]_{1,1} 
&= \sum_{u =1}^{p-1} t_{1,u }^2 + (p-1) 
= (p-1) + \sum_{u =1}^{p-1} 
\left[1 + \left( \frac{f_1(u )}{p} \right) \right]^2\\
&= p-2 + \sum_{x=0}^{p-1}
\left[1 + \left( \frac{f_1(x)}{p} \right) \right]^2  \\
&= (p-2) + p + 2\sum_{x=0}^{p-1}\left( \frac{f_1(x)}{p} \right) + \sum_{x=0}^{p-1}\left( \frac{f_1(x)}{p} \right)^2\\
&=2p - 2 +2(-1) + (p-2)  \\
&=3p-6.
\end{align*}
For $k=2,3,\ldots,p-1$, a similar computation and \eqref{eq:LegendreSum} yield
\begin{align*}
[T^2]_{1,k}
&=\sum_{u =1}^{p-1} t_{1,u } t_{u ,k}
=\sum_{u =1}^{p-1} \left[ 1 + \left( \frac{f_1(u )}{p} \right) \right]
\left[ 1 + \left( \frac{f_k(u )}{p} \right) \right] \\
&=-2 + \sum_{x=0}^{p-1} \left[ 1 + \left( \frac{f_1(x)}{p} \right) \right]
\left[ 1 + \left( \frac{f_k(x)}{p} \right) \right] \\
&=p-2 + \sum_{x=0}^{p-1} \left( \frac{f_1(x)}{p} \right) + \sum_{x=0}^{p-1} \left( \frac{f_k(x)}{p} \right) 
+ \sum_{x=0}^{p-1} \left( \frac{f_1(x)}{p} \right)\left( \frac{f_k(x)}{p} \right) \\
&= p-4 + \epsilon_k.
\end{align*}
For $k=p,p+1$, \eqref{eq:LegendreSum} and quadratic reciprocity provide
\begin{align*}
[T^2]_{1,k}
&= \sum_{u =2}^{p-1} t_{1,u }
= \sum_{u =2}^{p-1} \left[ 1 + \left( \frac{ f_1(u )}{p}\right)\right] \\
&= (p-2) -\left( \frac{f_1(0)}{p}\right)-\left( \frac{f_1(1)}{p}\right)
+ \sum_{x=0}^{p-1} \left( \frac{ f_1(x)}{p}\right) \\
&=(p-2) - \left( \frac{0}{p} \right) - \left( \frac{-3}{p} \right) -1 \\
&= p-3 - \left( \frac{-3}{p} \right) = p-3 - \ell_p.
\end{align*}
Finally,
\begin{align*}
[T^2]_{1,p+2}
&= t_{1,1} \sqrt{p-1} = \left[1 + \left( \frac{f_1(1)}{p}\right)\right] \sqrt{p-1} \\
&= \left[1 + \left( \frac{-3}{p} \right)\right] \sqrt{p-1} 
= (1+\ell_p) \sqrt{p-1} . \qedhere
\end{align*}
\end{proof}

We are now ready to complete the proof of Theorem \ref{Theorem:Main}.
Lemma \ref{Lemma:Entries} ensures that
\begin{equation*}
[T^2]_{1,k}^2
= 
\begin{cases}\displaystyle
9 p^2-36 p+36 & \text{if $k=1$},\\[5pt]
(p-4)^2 + 2(p-4) \epsilon_k + \epsilon_k^2 & \text{if $k=2,3,\ldots,p-1$},\\[5pt]
p^2 -p(6+2\ell_p ) + 10 + 6\ell_p  & \text{if $k = p,p+1$},\\[5pt]
2(1+\ell_p) (p-1)  & \text{if $k=p+2$}.
\end{cases}
\end{equation*}
A computation yields
\begin{align*}
[T^4]_{1,1}
&= \sum_{k=1}^{p+2} [T^2]_{1,k}^2 \\
&= (9 p^2-36 p+36 ) + \sum_{k=2}^{p-1}\big((p-4)^2 + 2(p-4) \epsilon_k + \epsilon_k^2\big)\\
&\qquad + 2\big(p^2 -p(6+2\ell_p  ) + 10 + 6\ell_p \big) + (1+\ell_p )^2 p - 2(1+\ell_p ) \\
&=p^3 + p^2 -2 (7+\ell_p )p + 2(11+5\ell_p ) + 2(p-4) \sum_{k=2}^{p-1} \epsilon_k + \sum_{k=2}^{p-1} \epsilon_k^2 \\
&= p^3 + p^2 -2 (7+\ell_p )p + 2(11+5\ell_p ) + 2(p-4) (\underbrace{4+\ell_p -p}_{\text{Lemma \ref{Lemma:EpsilonSum}}}) + \sum_{k=2}^{p-1} \epsilon_k^2 \\
&=p^3 - p^2 + 2p + 2(\ell_p -5) + \sum_{k=2}^{p-1} \epsilon_k^2 \\
&=p^3 - p^2 + 2p + 2(\ell_p -5) + \underbrace{(-1-\ell_p )^2}_{\text{Lemma \ref{Lemma:19}}}+ \sum_{\substack{k=2\\k\neq 9}}^{p-1} \epsilon_k^2 \\
&=p^3 - p^2 + 2p + 2(\ell_p -5) + 2(\ell_p +1) + \sum_{\substack{k=2\\k\neq 9}}^{p-1} \epsilon_k^2 \\
&=p^3 - p^2 + 2p + 4\ell_p  - 8+ \sum_{\substack{k=2\\k\neq 9}}^{p-1} \epsilon_k^2 .
\end{align*}
The definitions \eqref{eq:U} and \eqref{eq:D} imply that
\begin{equation*}
[UD^4U]_{1,1}
= \frac{1}{p^2} \bigg( \sum_{u =1}^{p-1} K_{u }^6 + 2 +(p-1)^5 \bigg) .
\end{equation*}
The equality $T^4 = UD^4U$ reveals that
\begin{equation*}
\sum_{u =1}^{p-1} K_{u }^6 + 2 +(p-1)^5 = p^2\bigg(p^3 - p^2 + 2p + 4 \ell_p - 8+ \sum_{\substack{k=2\\k\neq 9}}^{p-1} \epsilon_k^2 \bigg).
\end{equation*}
Consequently,
\begin{align*}
\sum_{u =1}^{p-1} K_{u }^6
&= \bigg( p^5 - p^4 + 2p^3 + (4\ell_p-8)p^2 + p^2 \sum_{\substack{k=2\\k\neq 9}}^{p-1} \epsilon_k^2 \bigg) - 2 - (p-1)^5 \\
&= 4 p^4 -8p^3+ (4 \ell_p+2) p^2-5 p-1  + p^2 \sum_{\substack{k=2\\k\neq 9}}^{p-1} \epsilon_k^2\\
&= 4 p^4 -8p^3+ (4 \ell_p+2) p^2-5 p-1  + p^2 \sum_{\substack{k=2\\k\neq 9}}^{p-1} (a_p(E_k)+1)^2.
\end{align*}
This is the desired formula \eqref{eq:V6}. \qed

\section{Future work}\label{FutureWork}

It follows from \eqref{eq:U} and \eqref{eq:D} that the $n$th Kloosterman power moment is 
\begin{align*}
    V_n(p) = p^2 [T_1]_{1,1}^{n-2} + 2(-1)^{n-1} - (p-1)^{n-1}.
\end{align*}
Thus, the problem of calculating $V_n(p)$ reduces to the evaluation of sums and products of Legendre symbols. 
For the sixth power moment, Lemmas \ref{Lemma:EpsilonSum} and \ref{Lemma:Transform} allowed us to provide an evaluation 
in terms of power moments of Frobenius traces. For higher moments, similar, but more complicated, 
techniques are needed.  There is much work to be done in this direction.

In \cite{Kaplan}, Kaplan and Petrow provide a method to evaluate power moments of Frobenius traces of families of elliptic
curves whose group of $k$-points contains a particular subgroup. Their evaluation is in terms of traces of Hecke operators. 
Given that is possible to relate the constants appearing in the evaluation of the fifth through eighth power moments of Kloosterman sums in terms of Hecke operators, one wonders if the Kaplan--Petrow method can be used to evaluate the power moment of the Frobenius traces that appear in Theorem \ref{Theorem:Main} and whether or not such terms will appear in higher power moments when evaluated using \eqref{eq:U} and \eqref{eq:D}. In particular, it follows from \cite{Livne, Peters}, that $V_5(p)$ can be expressed in terms of Hecke eigenvalues for a weight $3$ newform on $\Gamma_0(15)$. That $V_6(p)$ can be expressed in terms of Hecke eigenvalues for a weight $4$ newform on $\Gamma_0(6)$ follows from \cite{Hulek}. 
Evans conjectured that $V_7(p)$ and $V_8(p)$ can be evaluated in terms of Hecke eigenvalues for a weight $3$ newform on $\Gamma_0(525)$ and for a weight $6$ newform on $\Gamma_0(525)$, respectively \cite{Evans, Evans2}. Yun proved Evans' conjectures in \cite{Yun2}.

Mixed Kloosterman moments are also of interest and have been studied in \cite{Chavez, CKS, Kutzko, Lehmer}. 
From \eqref{eq:U} and \eqref{eq:D}, we have
\begin{equation}\label{eq:Mixed}
\sum_{u=1}^{p-1} K_uK_{a_1u}K_{a_2u}\cdots K_{a_nu}=p^2\left[T_{a_1}T_{a_2}\cdots T_{a_{n-1}}\right]_{1,a_n}+2(-1)^n-(p-1)^n.
\end{equation}
The second and third mixed moments are given by
\begin{align*}
    \sum_{u=1}^{p-1} K_u K_{au} &= -p ,\\
    \sum_{u=1}^{p-1} K_u K_{au} K_{bu} &= \left( \frac{f_a(b)}{p}\right)p^2 + 2p;
\end{align*}
see \cite{Lehmer, Kutzko, CKS}.
In \cite{Chavez}, \'A.~Ch\'avez and the second author showed that
\begin{align*}
	\sum_{u=1}^{p-1} K_uK_{au}K_{bu}K_{cu} = \delta_{a,1}\delta_{b,c}p^3 - \left[ \left(\frac{bc}{p}\right)a_p + 2\right]p^2 - 3p - 1,
\end{align*}
in which $a_p$ is a certain Frobenius trace. In light of \eqref{eq:Mixed}, a similar evaluation for higher mixed moments appears within reach.

\bibliography{K6M}{}

\begin{thebibliography}{10}

\bibitem{Andre1}
Carlos A.~M. Andr\'e.
\newblock The basic character table of the unitriangular group.
\newblock {\em J. Algebra}, 241(1):437--471, 2001.

\bibitem{SESUP}
J.~L. Brumbaugh, Madeleine Bulkow, Patrick~S. Fleming, Luis~Alberto
  Garcia~German, Stephan~Ramon Garcia, Gizem Karaali, Matt Michal, Andrew~P.
  Turner, and Hong Suh.
\newblock Supercharacters, exponential sums, and the uncertainty principle.
\newblock {\em J. Number Theory}, 144:151--175, 2014.

\bibitem{ChahalMonthly}
Jasbir~S. Chahal and Brian Osserman.
\newblock The {R}iemann hypothesis for elliptic curves.
\newblock {\em Amer. Math. Monthly}, 115(5):431--442, 2008.

\bibitem{Chavez}
\'Angel Ch\'avez and George Todd.
\newblock Supercharacters and mixed moments of {K}loosterman sums.
\newblock {\em International Journal of Number Theory}.
\newblock in press.

\bibitem{CR62}
Charles~W. Curtis and Irving Reiner.
\newblock {\em Representation theory of finite groups and associative
  algebras}.
\newblock Pure and Applied Mathematics, Vol. XI. Interscience Publishers, a
  division of John Wiley \& Sons, New York-London, 1962.

\bibitem{Davenport}
H.~Davenport.
\newblock On certain exponential sums.
\newblock {\em J. Reine Angew. Math.}, 169:158--176, 1933.

\bibitem{Diaconis}
Persi Diaconis and I.~M. Isaacs.
\newblock Supercharacters and superclasses for algebra groups.
\newblock {\em Trans. Amer. Math. Soc.}, 360(5):2359--2392, 2008.

\bibitem{Evans2}
Ron Evans.
\newblock Hypergeometric {$_3F_2(1/4)$} evaluations over finite fields and
  {H}ecke eigenforms.
\newblock {\em Proc. Amer. Math. Soc.}, 138(2):517--531, 2010.

\bibitem{Evans}
Ronald Evans.
\newblock Seventh power moments of {K}loosterman sums.
\newblock {\em Israel J. Math.}, 175:349--362, 2010.

\bibitem{CKS}
Patrick~S. Fleming, Stephan~Ramon Garcia, and Gizem Karaali.
\newblock Classical {K}loosterman sums: representation theory, magic squares,
  and {R}amanujan multigraphs.
\newblock {\em J. Number Theory}, 131(4):661--680, 2011.

\bibitem{RSS}
Christopher~F. Fowler, Stephan~Ramon Garcia, and Gizem Karaali.
\newblock Ramanujan sums as supercharacters.
\newblock {\em Ramanujan J.}, 35(2):205--241, 2014.

\bibitem{GHM}
Stephan~Ramon Garcia, Trevor Hyde, and Bob Lutz.
\newblock Gauss's hidden menagerie: from cyclotomy to supercharacters.
\newblock {\em Notices Amer. Math. Soc.}, 62(8):878--888, 2015.

\bibitem{Hasse}
Helmut Hasse.
\newblock Zur {T}heorie der abstrakten elliptischen {F}unktionenk\"orper {III}.
  {D}ie {S}truktur des {M}eromorphismenrings. {D}ie {R}iemannsche {V}ermutung.
\newblock {\em J. Reine Angew. Math.}, 175:193--208, 1936.

\bibitem{Hulek}
K.~Hulek, J.~Spandaw, B.~van Geemen, and D.~van Straten.
\newblock The modularity of the {B}arth-{N}ieto quintic and its relatives.
\newblock {\em Adv. Geom.}, 1(3):263--289, 2001.

\bibitem{Ireland}
Kenneth Ireland and Michael Rosen.
\newblock {\em A classical introduction to modern number theory}, volume~84 of
  {\em Graduate Texts in Mathematics}.
\newblock Springer-Verlag, New York, second edition, 1990.

\bibitem{Kaplan}
Nathan Kaplan and Ian Petrow.
\newblock Elliptic curves over a finite field and the trace formula.
\newblock {\em Proc. Lond. Math. Soc. (3)}, 115(6):1317--1372, 2017.

\bibitem{Kutzko}
Philip~C. Kutzko.
\newblock The cyclotomy of finite commutative {P}.{I}.{R}.'s.
\newblock {\em Illinois J. Math.}, 19:1--17, 1975.

\bibitem{Lehmer}
D.~H. Lehmer and Emma Lehmer.
\newblock The cyclotomy of {K}loosterman sums.
\newblock {\em Acta Arith.}, 12:385--407, 1966/67.

\bibitem{Lisonek}
Petr Lison\v{e}k.
\newblock On the connection between {K}loosterman sums and elliptic curves.
\newblock In {\em Sequences and their applications---{SETA} 2008}, volume 5203
  of {\em Lecture Notes in Comput. Sci.}, pages 182--187. Springer, Berlin,
  2008.

\bibitem{Livne}
Ron Livn\'e.
\newblock Motivic orthogonal two-dimensional representations of {${\rm
  Gal}(\overline {\bf Q}/\bold Q)$}.
\newblock {\em Israel J. Math.}, 92(1-3):149--156, 1995.

\bibitem{Manin}
Yu.~I. Manin.
\newblock On cubic congruences to a prime modulus.
\newblock {\em Izv. Akad. Nauk SSSR. Ser. Mat.}, 20:673--678, 1956.

\bibitem{Peters}
C.~Peters, J.~Top, and M.~van~der Vlugt.
\newblock The {H}asse zeta function of a {$K3$} surface related to the number
  of words of weight {$5$} in the {M}elas codes.
\newblock {\em J. Reine Angew. Math.}, 432:151--176, 1992.

\bibitem{Salie}
Hans Sali\'e.
\newblock \"uber die {K}loostermanschen {S}ummen {$S(u,v;q)$}.
\newblock {\em Math. Z.}, 34(1):91--109, 1932.

\bibitem{Weil}
Andr{\'e} Weil.
\newblock On some exponential sums.
\newblock {\em Proc. Nat. Acad. Sci. U. S. A.}, 34:204--207, 1948.

\bibitem{WilliamsPAMS}
Kenneth~S. Williams.
\newblock Evaluation of character sums connected with elliptic curves.
\newblock {\em Proc. Amer. Math. Soc.}, 73(3):291--299, 1979.

\bibitem{XiYi}
Ping Xi and Yuan Yi.
\newblock A note on the moments of {K}loosterman sums.
\newblock {\em Proc. Amer. Math. Soc.}, 141(4):1233--1240, 2013.

\bibitem{Yun2}
Zhiwei Yun.
\newblock Galois representations attached to moments of {K}loosterman sums and
  conjectures of {E}vans.
\newblock {\em Compos. Math.}, 151(1):68--120, 2015.
\newblock Appendix B by Christelle Vincent.

\bibitem{ZhangHan}
Wenpeng Zhang and Di~Han.
\newblock A new identity involving the classical {K}loosterman sums and
  2-dimensional {K}loosterman sums.
\newblock {\em Int. J. Number Theory}, 12(1):111--119, 2016.

\end{thebibliography}
\bibliographystyle{plain}

\end{document}